\documentclass[11pt
]{amsart}

\usepackage{mathptmx} 
\usepackage[scaled=0.90]{helvet} 
\usepackage{courier} 
\normalfont
\usepackage[T1]{fontenc}

\usepackage{hyperref}
\hypersetup{bookmarksdepth=3}

\usepackage{geometry}
\usepackage{amsmath,amssymb}
\usepackage{
mathrsfs}
\usepackage{esint}

\usepackage{tensor}

\usepackage[show]{ed}

\allowdisplaybreaks[1]

    \normalbaselines                          %

\newtheorem{thm}{Theorem}[section]
\newtheorem{lm}[thm]{Lemma}

\theoremstyle{definition}

\newtheorem*{df*}{Definition}

\theoremstyle{remark}

\newtheorem*{rem*}{Remark}

\numberwithin{equation}{section}

\newcommand{\ci}[1]{_{ {}_{\scriptstyle #1}}}

\newcommand{\al}{\alpha}

\newcommand{\cz}{Calder\'{o}n--Zygmund\ }

\newcommand{\E}{\mathbb{E}}

\newcommand{\R}{\mathbb{R}}

\newcommand{\QQ}{[w]_{A_2}}

\newcommand{\wt}{\widetilde}
\newcommand{\La}{\langle}
\newcommand{\Ra}{\rangle}

\newcommand{\cD}{\mathscr{D}}

\newcommand{\cP}{\mathcal{P}}

\newcommand{\cE}{\mathcal{E}}

\def\cyr{\fontencoding{OT2}\fontfamily{wncyr}\selectfont}
\DeclareTextFontCommand{\textcyr}{\cyr}
\newcommand{\sha}[0]{\ensuremath{\mathbb{S}
}}

{\end{list}}

\newcounter{vremennyj}

\begin{document}

\title[Simple estimate of shifts]{A simple sharp weighted estimate of the dyadic shifts on metric spaces with geometric doubling}
\author{Fedor Nazarov}
\address{Dept. of Mathematics, Univ. of Wisconsin, Madison}
\author{Alexander Volberg}
\address{Department of Mathematics, Michigan State University, East
Lansing, MI 48824, USA}

\date{}

\begin{abstract}
We give a short and simple polynomial estimate of the norm of weighted dyadic shift on metric space with geometric doubling, which is linear in the norm of the weight. Combined with the existence of special  probability space of dyadic lattices built in A. Reznikov, A. Volberg, `` Random ``dyadic'' lattice in geometrically doubling metric space and $A_2$ conjecture", arXiv:1103.5246, and with decomposition of \cz operators to dyadic shifts from 
Hytonen's
``The sharp weighted bound for general Calderon-Zygmund operators", arXiv:1007.4330 (and later T. Hyt\"onen, C. P\'erez, S. Treil, A. Volberg, ``A sharp estimated of weighted dyadic shifts that gives the proof of $A_2$ conjecture",  arXiv 1010.0755.), we will be able to have a linear (in the norm of weight) estimate of an arbitrary \cz operator on a metric space with geometric doubling. This will be published separately.
\end{abstract}

\maketitle

\section{Introduction}
\label{Intro}

Recall that in \cite{PTV1} it was proved that
\begin{thm}
\label{weak}
If $T$ is an arbitrary operator with a \cz kernel, then
\begin{align*}
\|T\|_{L^2(w\,d\mu)\rightarrow L^{2,\infty}(w)}+ & \|T'\|_{L^2(w^{-1})\rightarrow L^{2,\infty}(w^{-1})}
 \le 2\|T\|_{L^2(w\,d\mu)\rightarrow L^2(w\,d\mu)} 
\\
& \le C\,(\QQ+ \|T\|_{L^2(w\,d\mu)\rightarrow L^{2,\infty}(w)}+ \|T'\|_{L^2(w^{-1})\rightarrow L^{2,\infty}(w^{-1})}).
\end{align*}
\end{thm}
By $T'$ we denote the adjoint operator. Here of course only the right inequality is interesting. And it is unexpected too. The weak and strong norm of any operator with a \cz kernel turned out to be equivalent up to additive term $\QQ$.  In its turn, Theorem \ref{weak} was obtained in \cite{PTV1} as a corollary of a weighted $T1$ theorem--the Main Theorem of \cite{PTV1}. Again in its turn the Main Theorem (={\it weighted} $T1$ {\it theorem}) in \cite{PTV1} is a  consequence of a rather difficult two-weight $T1$ theorem of \cite{NTVlost}. From Theorem \ref{weak} we obtained in \cite{PTV1} the following result which holds for any \cz operator.
\begin{thm}
\label{A2log}
$\|T\|_{L^2(w\,d\mu)\rightarrow L^2(w\,d\mu)} \le C\cdot \QQ\log (1+\QQ)$.
\end{thm}
By $A_2$ conjecture people understand the strengthening of this claim, where the logarithmic term is deleted, in other words, a linear  (in weight's norm) estimate of  arbitrary weighted \cz operator.
In \cite{HLRSVUT} the $A_2$ conjecture was proved for \cz operators having more than $2d$ smoothness in $\mathbb{R}^d$. 

A bit later a preprint \cite{H} of Tuomas Hyt\"onen has appeared, the $A_2$ conjecture is fully proved there.  It is based on the Main Theorem (={\it weighted} $T1$ {\it theorem}) in \cite{PTV1} of P\'erez--Treil--Volberg.  Both \cite{PTV1} and \cite {H} are neither short nor easy. 

Notice that the scheme of the proof in \cite{H}  goes like that: given a \cz operator $T$ and a $w\in A_2$ weight, one first uses the Main Theorem (={\it weighted} $T1$ {\it theorem}) of \cite{PTV1}, which says that to prove a linear estimate for $\|Tf\|_{w}$ it is enough to prove it uniformly only for special ``characteristic functions of cubes" type functions $f$ (see the true statement in \cite{PTV1}). Notice that all the cubes must be checked. The second (very beautiful) part of the proof is to decompose $T$ into ``a convex combinations" of dyadic shifts, the new idea is used here that grew out of random lattices approach in non-homogeneous Harmonic Analysis theory of Nazarov--Treil--Volberg, see \cite{NTV5}, \cite{NTV7}, \cite{NTVlost}, \cite{Vo1}. Now it is enough to check the right estimate for each dyadic shift applied to each ``characteristic function of cube". A very annoying difficulty appears: the shift is with respect to a certain dyadic lattice, but the cube in question is arbitrary and a priori does not belong to this lattice. This creates serious technical difficulties, which can be (and were) avoided in \cite{HPTV}. 

 The direct proof of $A_2$ conjecture (without going through \cite{PTV1}) was given in \cite{HPTV}, and it was based on two ingredients: 1) a formula for decomposing an arbitrary \cz operators into (generalized) dyadic shifts by the averaging trick, see \cite{H}, 2) on a polynomial in complexity and linear in weight estimate of the norm of a dyadic shift. 
 
 The latter was quite complicated and was based on modification of the argument in Lacey--Petermichl-Reguera \cite{LPR} and on the use of \cite{NTV6} with its careful reexamination. The former--as we already mentioned-- was rooted in the works on non-homogeneous Harmonic Analysis, like e. g. \cite{NTV5}--\cite{NTV7}, \cite{NTVlost},  \cite{Vo1}, but with  a new twist, which appeared first in Hyt\"onen's \cite{H} and was simplified in  Hyt\"onen--P\'erez--Treil--Volberg's \cite{HPTV}. 
 
 The averaging trick was a development of the bootstrapping argument used by Nazarov--Treil--Volberg, where they exploited the fact that the bad part of a function can be made arbitrarily small. Using the original Nazarov--Treil--Volberg averaging trick would add an extra factor depending on $[w]\ci{A_2}$ to the estimate, so a new idea was necessary. The new observation in \cite{H} was that as soon as the probability of a ``bad'' cube is less than $1$, it is possible to completely ignore the bad cubes (at least in the situation where they cause troubles).

\section{Shifts of complexity $m,n$}
\label{shifts}

Let $X$ be a compact geometrically doubling metric space. Let $\mu$ be a doubling measure on $X$, which exists by \cite{KV}.  Let $\cD$ be a finite ``dyadic" lattice on $X$. Namely, $\cD$ consists of disjoint partition $\cE_N$ to pieces $Q_N^j$ of size $\asymp \delta^N$, then there is a partition $\cE_{N-1}$ to pieces $Q_{N-1}^i$ of size $\delta^{N-1}$, each $Q_{N-1}^i$ consists of boundedly many (at most $M_{N-1}(i)\le M$) pieces of partition $\cE_N$, et cetera... , we have $\cE_k, k=N, N-1,...,0$, and each $Q_k^i$ is an almost ball:
if $\ell(Q), Q=Q_k^i$, denote its diameter, then there is a ball of radius $c\ell(Q)$, $c>0$, inside $Q$, and $c$ does not depend on $N, k, i$.

The existence of such lattices with all constants depending only on geometric doubling of $X$ was proved by Christ \cite{Chr}.

We relate to $\cE_k$ the projection operator $\E_k$ on $L^2(\mu)$: $\E_k f= \sum_j \langle f\rangle_{Q_k^j, \mu} \chi_{Q_k^j}$.
We also consider the martingale difference operator $\Delta_k f = E_k f- E_{k-1} f$.  Notice that it can be written as
$$
\Delta_k f = \sum_{i}\sum_{j=1}^{M_{k-1}(Q_{k-1}^i)} (f, h_{Q_{k-1}^i}^j)_{\mu} h_{Q_{k-1}^i}^j\,,
$$
where denoting $Q=Q_{k-1}^i$ we notice that  $M_{k-1}(Q)\le M$,  and $ h_{Q}^j$,  are functions supported on $Q$, orthogonal to constants in $L^2(\mu)$, orthogonal to each other in $L^2(\mu)$,  constant on each $Q_k^i\subset Q$, $Q_k^i\in \cE_k$ (such $Q_k^i$ are called the {\bf sons} of $Q_{k-1}^j$), and
having the following bound
\begin{equation}
\label{Linfty}
\|h_Q^j\|_{L^\infty(\mu)} \le \frac{C}{\mu(Q)^{1/2}}\,.
\end{equation}

\noindent{\bf Definition.} We call such $h_{Q_k}^j$ Haar functions. If $L\in \cE_m, J\in \cE_{m+n}$ we say $g(J)=g(L)+n$.

\bigskip 

We always use $\ell(S)$ to denote the diameter of a set $S$. Christ's cubes will be denoted $Q, I, J, L, K$, may be with indices.

\bigskip

We call by $\sha_{m,n}$ the operator given by the  kernel
$$
f\rightarrow \sum_{L\in \cD} \int_L a_L(x,y)f(y)dy\,,
$$
where 
$$
a_L(x,y) =\sum_{\substack{ I\subset L, J\subset L\\ g(I)= g(L)+m, \, g(J)= g(L)+n}}c_{L,I,J} h_J^j(x)h_I^i(y)\,,
$$
where $h_I^i, h_J^j$ are Haar functions normalized in $L^2(d\mu)$ and satisfying \eqref{Linfty}, and $|c_{L,I,J}|\le \frac{\sqrt{\mu(I)}\sqrt{\mu(J)}}{\mu(L)}$. 
Often we will skip superscripts $i, j$.

\bigskip

We are interested in sufficiently good estimate of  
$$
\|\sha_{m,n}\|_w:=\|\sha_{m,n}:L^2(w\,d\mu)\rightarrow L^2(w\,d\mu)\|\,,
$$
where $w\in A_2$. For such $w$ we put $\sigma=w^{-1}$ and
$$
[w]_{A_2}:=\sup_I\La w\Ra_{\mu, I}\La \sigma\Ra_{\mu, I}<\infty\,,
$$
and call it the norm of $w$ (it is not a norm).

In recent paper \cite{HPTV} the following theorem was proved ( another proof, using the Bellman function technique, was given recently in \cite{NV})
\begin{thm}
\label{poly1}
\begin{equation}
\label{polyeq1}
\|\sha_{m,n}\|_w \le C\, (m+n+1)^a [w]_{A_2}\,.
\end{equation}
\end{thm}

In \cite{HPTV} $a=3$. Looks like here we have the same numerical value. But for its main application in \cite{HPTV}: the proof of $A_2$ conjecture, the value of $a$ (if finite) is not important.
The proof was hard and combinatorial, it was based on the ideas of \cite{LPR}, where such an estimate was proved with exponential dependence on $m+n$. We propose here a simple proof based on Bellman function technique. This  technique was  tried successfully for shifts of low complexity, first in \cite{Wit},  \cite{PetmV} then in \cite{Petm1}, \cite{Petm2}, and recently in preprint \cite{RTV}, which gives a simple unified proof of results in  \cite{Wit},   \cite{Petm1}, \cite{Petm2}. As the reader will see one needs a couple of new tricks to achieve this fulfillment.
The proof below is a direct and simple readjustment of the proof in \cite{NV}, where it has been carried out if $X=\R$, $\mu=dx$.

\bigskip

\noindent{\bf Remark.} The reasoning below is in $\R$. But one can modify it without any efforts to any $\R^d$. Moreover, in \cite{RV} the probability space of Christ's type dyadic lattices is built on any compact metric space with the property of geometric doubling (every ball  contains at most a fixed number of disjoint balls of half a radius), which allows to extend the sharp bound of \cz operators into metric space setting  by repeating the averaging trick that reduces everything to the case of dyadic shift on the metric space, and then using this preprint to give a polynomial in complexity and linear in weight estimate for any shift.

\section{The heart of the matter: a reduction to bilinear embedding estimate}
\label{bilin}

To prove Theorem \ref{poly1} we need the  following decomposition:
\begin{lm}
\label{decomp}
$$
h_I^j = \al_I ^jh_I^{w,j} + \beta_I^j \chi_I\,,
$$
where

1)  $|\alpha_I^j| \le \sqrt{\langle w\rangle_{\mu, I}}$,

2)$ |\beta_I^j| \le \frac{(h_I^{ w,j}, w)_{\mu}}{w(I)}$, where $w(I):= \int_I w\,d\mu$,

3) $\{h_I^{w,j}\}_{I} $ is supported on $I$, orthogonal to constants in $L^2(w\,d\mu)$,

4) $h_I^{w,j}$ assumes on each son $s(I)$ a constant values,

5) $\|h_I^{w,j}\|_{L^2(\mu)}=1$.
\end{lm}

\bigskip

\noindent{\bf Definition.}
Let 
$$
\Delta_I w:= \sum_{\text{sons of}\,\,I} |\langle w\rangle_{\mu, s(I)}-\langle w\rangle_{\mu, I}|\,.
$$
\noindent{\bf Remark.} There are many $\Delta$'s in this paper, but the reader should notice that $\Delta_k$ is an operator sending functions from $L^2(\mu)$ to such functions with extra properties of being constants on each element of $\cE_k$ and also being orthogonal in $L^2(\mu)$ to $\chi_L$ for any element $L$ of $\cE_{k-1}$. On the other hand, $\Delta_I w$ (as well as $\Delta_J\sigma $) is a non-negative number.

It is a easy to see (by a small linear algebra reasoning)  that the doubling property of measure $\mu$ implies
\begin{equation}
\label{delta}
|(h_I^{ w,j}, w)_{\mu}|\le C\,\cdot (\Delta_I w)\,\sqrt{\mu(I)}\,.
\end{equation}
Here $C$ depends only on the doubling constant of measure $\mu$.  In other words, taking into account that $\Delta_k w =\sum_{I\in \cE_{k-1}}\sum_{j=1}^{M_{k-1}(I)} (w, h_I^{w,j})_{\mu} h_I^{w,j}$ and given $I\in \cE_{k-1}$, we can rewrite \eqref{delta} as follows
\begin{equation}
\label{delta1}
\chi_I |\Delta_k w| \le C\, \cdot   (\Delta_I w)\, ,
\end{equation}
where $C$ again depends only on the doubling constant of $\mu$.

Therefore, the property 2) above can be rewritten as 

\medskip

2') $|\beta_I^j|\le C\,\frac{|\Delta_I w|}{\langle w\rangle_{\mu, I}} \frac{1}{\sqrt{\mu (I)}}$.

\vspace{.2in}

Fix $\phi\in L^2(w\,d\mu), \psi\in L^2(\sigma)$. We need to prove
\begin{equation}
\label{main}
|(\sha_{m,n}\phi w,\psi\sigma)|\le C\, (n+m+1)^a\|\phi\|_{w}\|\psi\|_{\sigma}\,.
\end{equation}

We  estimate $(\sha_{m,n}\phi w, \psi \sigma)$ as (we skip superscripts $j, i$ and write just $h_J, h_I$):

$$
|\sum_L\sum_{ I, J} c_{L,I,J}(\phi w, h_I)_{\mu}(\psi \sigma, h_J)_{\mu}|\le 
$$
$$
\sum_L\sum_{ I,J} |c_{L,I,J} (\phi w, h^w_I)_{\mu}\sqrt{\langle w \rangle_{\mu, I}}(\psi \sigma,h^{\sigma}_J)_{\mu}|\sqrt{\langle \sigma \rangle_{\mu, J}}|\,+
$$
$$
\sum_L\sum_{I,J} |c_{L,I,J} \langle \phi w\rangle_{\mu, I}\frac{\Delta_I w}{\langle w \rangle_{\mu, I}}(\psi \sigma,h^{\sigma}_J)_{\mu}\sqrt{\langle \sigma \rangle_{\mu, J}}\sqrt{\mu(I)}|\,+
$$
$$
\sum_L\sum_{I,J} |c_{L,I,J} \langle \psi \sigma\rangle_{\mu, J}\frac{\Delta_J\sigma}{\langle \sigma \rangle_{\mu, J}}(\phi w,h^{w}_I)_{\mu}\sqrt{\langle w \rangle_{\mu, I}}\sqrt{\mu(J)}|\,+
$$
$$
\sum_L\sum_{I,J}|c_{L,I,J} \langle \phi w\rangle_{\mu, I}\langle \psi \sigma\rangle_{\mu, J} \frac{\Delta_I w}{\langle w \rangle_{\mu, I}} \frac{\Delta_J\sigma}{\langle \sigma \rangle_{\mu, J}}\sqrt{\mu(I)}\sqrt{\mu(J)}|=: I + II +III +IV\,.
$$

We can notice  that  because we have $|c_{L,I,J}|\le \frac{\sqrt{\mu(I)}\sqrt{\mu(J)}}{\mu(L)}$,  each sum inside $L$ can be estimated  by a perfect product of $S$ and $R$ terms, where
$$
R_L(\phi w):= \sum_{I\subset L...} \langle \phi w\rangle_{\mu, I}  \frac{|\Delta_I w|}{\langle w \rangle_{\mu, I}}\frac{\mu(I)}{\sqrt{\mu(L)}}
$$
$$
S_L(\phi w) := \sum_{I\subset L...} (\phi w, h^w_I)_{\mu}\sqrt{\langle w \rangle_{\mu, I}}\frac{\sqrt{\mu(I)}}{\sqrt{\mu(L)}}
$$
and the corresponding terms for $\psi \sigma$.
So we have
$$
I \le \sum_L S_L(\phi w)S_L(\psi\sigma)
,\,
II \le \sum_L S_L(\phi w)R_L(\psi\sigma),\,\,
\,
$$
$$
III\le \sum_L R_L(\phi w)S_L(\psi\sigma)
,\,\,\,
IV\le \sum_L R_L(\phi w)R_L(\psi\sigma)\,.
$$
Now
\begin{equation}
\label{sl}
S_L(\phi w) \le \sqrt{\sum_{I\subset L...} |(\phi w, h^w_I)_{\mu}|^2}\sqrt{\La w\Ra_{\mu, L}}\,,\,\,\,S_L(\psi \sigma) \le \sqrt{\sum_{J\subset L...} |(\psi \sigma, h^\sigma_J)|^2}\sqrt{\La \sigma\Ra_{\mu, L}}
\end{equation}
Therefore,
\begin{equation}
\label{I}
I\le C\QQ^{1/2} \|\phi\|_w\|\psi\|_{\sigma}\,.
\end{equation}

Terms $II, III$ are symmetric, so consider $III$.
Using Bellman function $(xy)^{\al}$ one can prove now 
\begin{lm}
\label{uval}
The sequence 
$$
\mu_I := \La w\Ra_{\mu, I}^{\al}\La \sigma\Ra_{\mu, I}^{\al}\bigg(\frac{|\Delta_I w|^2}{\La w \Ra_{\mu, I}^2} + \frac{|\Delta_I \sigma|^2}{\La \sigma \Ra_{\mu, I}^2}\bigg) \mu(I)
$$
form a Carleson measure with Carleson constant  at most $c_{\al}Q^{\al}$, where $Q:=[w]_{A_2}$ for any $\al\in (0, 1/2)$. Here $c_{\al}$ depends only on $\al$ and the doubling constant of $\mu$ (and is independent of the doubling constants of $w\,d\mu$, $\sigma\,d\mu$).
\end{lm}
\begin{proof}
We need a very simple 

\medskip

\noindent{\bf Sublemma}.
Let $Q> 1, 0<\alpha<\frac12$. In domain  $\Omega_Q:=\{(x,y): X>o, y>0, 1<xy\le Q$ function $B_Q(x,y):=x^{\al}y^{\al}\}$ satisfies the following estimate of its Hessian matrix  (of its second differential form, actually)
$$
-d^2 B_Q(x,y)\ge \al(1-2\al)x^{\al}y^{\al}\bigg(\frac{(dx)^2}{x^2} +\frac{(dy)^2}{y^2}\bigg)\,.
$$
The form $-d^2 B_Q(x,y)\ge 0$ everywhere in $x>0, y>0$. Also obviously $0\le B_Q(x,y) \le Q^{\al}$ in $\Omega_Q$.
\begin{proof}
Direct calculation.
\end{proof}

\medskip

Fix now a Christ's cube $I$ and let $s_i(I), i=1,...,M$, be all its sons. Let $a=(\La w\Ra_{\mu, I}, \La \sigma\Ra_{\mu, I})$, $b_i=(\La w\Ra_{\mu,s_i( I)}, \La \sigma\Ra_{\mu, s_i(I)})$, $i=1,\dots, M$, be points--obviously--in $\Omega_Q$, where $Q$ temporarily means $[w]_{A_2}$. Consider $c_i(t)=a(1-t)+ b_it, 0\le t \le 1$ and $q_i(t):= B_Q(c_i(t))$.  We want to use Taylor's formula
\begin{equation}
\label{Lagr}
q_i(0)-q_i(1) = -q'_i(0) - \int_0^1dx\int_0^x q_i''(t)\,dt\,.
\end{equation}
Notice two things: Sublemma  shows that $-q_i''(t) \ge 0$ always. Moreover, it shows that if $t\in [0,1/2]$,  then we have that the following qualitative estimate holds:
\begin{equation}
\label{wI}
-q_i''(t) \ge c\,( \La w\Ra_{\mu, I} \La \sigma\Ra_{\mu, I})^{\al}\bigg(\frac{(\La w\Ra_{\mu,s_i( I)}-\La w\Ra_{\mu, I})^2}{\La w\Ra_{\mu, I}^2} +\frac{(\La \sigma\Ra_{\mu,s_i( I)}-\La \sigma\Ra_{\mu, I})^2}{\La \sigma\Ra_{\mu, I}^2} \bigg)
\end{equation}
This requires a small explanation. If we are on the segment $[a, b_i]$, then the first coordinate of such a point cannot be larger than $C\, \La w\Ra_{\mu, I}$, where $C$ depends only on doubling of $\mu$ (not $w$). This is obvious. The same is true for the second coordinate with the obvious change of $w$ to $\sigma$. But there is no such type of estimate from below on this segment:  the first coordinate cannot be smaller than $k\, \La w\Ra_{\mu, I}$, but $k$ may (and will) depend on the doubling of $w$ (so ultimately on its $[w]_{A_2}$ norm. In fact, at the ``right" endpoint of $[a, b_i]$. The first coordinate is $\La w\Ra_{\mu, s_i(I)}\le \int_I\,w\,d\mu/ \mu(s_i(I)) \le C\,  \int_I\,w\,d\mu/ \mu(I)=C\, \La w\Ra_{\mu, I}$, with $C$ only depending on the doubling of $\mu$. But the estimate from below will involve the doubling of $w$, which we must avoid. But if $t\in [0,1/2]$, and we are on the ``left half" of interval $[a, b_i]$ then obviously the first coordinate is $\ge \frac12 \La w\Ra_{\mu, I}$ and the second coordinate is $\ge \frac12 \La \sigma\Ra_{\mu, I}$.

We do not need to integrate $-q_i''(t)$ for all $t\in [0,1]$ in \eqref{Lagr}. We can only use integration over $[0,1/2]$  noticing that $-q_i''(t)\ge 0$ otherwise. Then the chain rule 
$$
q_i''(t)=(B_Q(c_i(t))''=(d^2B_Q(c_i(t)) (b_i-a), b_i-a)\,,
$$ 
(where $(\cdot, \cdot)$ means the usual scalar product in $\R^2$) immediately gives us \eqref{wI} with constant $c$ depending on the doubling of $\mu$ but {\it independent} of the doubling of $w$.

\medskip

Next step is to add all \eqref{Lagr}, with convex coefficients $\frac{\mu(s_i(I))}{\mu(I)}$, and to notice that $\sum_{i=1}^M\frac{\mu(s_i(I))}{\mu(I)} q_i'(0) =\nabla B_{Q}(a) \sum_{i=1}^M\cdot (a-b_i)\frac{\mu(s_i(I))}{\mu(I)}=0$, because by definition
$$
a= \sum_{i=1}^M \frac{\mu(s_i(I))}{\mu(I)}\,b_i\,.
$$
Notice that the addition of all \eqref{Lagr}, with convex coefficients $\frac{\mu(s_i(I))}{\mu(I)}$ gives us now ( we take into account \eqref{wI} and positivity of $-q_i''(t)$)
$$
B_Q(a)- \sum_{i=1}^M \frac{\mu(s_i(I))}{\mu(I)}\,B_Q(b_i) \ge c\,c_1\,( \La w\Ra_{\mu, I} \La \sigma\Ra_{\mu, I})^{\al}\sum_{i=1}^M\bigg(\frac{(\La w\Ra_{\mu,s_i( I)}-\La w\Ra_{\mu, I})^2}{\La w\Ra_{\mu, I}^2} +\frac{(\La \sigma\Ra_{\mu,s_i( I)}-\La \sigma\Ra_{\mu, I})^2}{\La \sigma\Ra_{\mu, I}^2} \bigg)\,.
$$
We used here the doubling of $\mu$ again, by noticing that $\frac{\mu(s_i(I))}{\mu(I)}\ge c_1$ (recall that $s_i(I)$ and $I$ are almost balls of comparable radii).
We rewrite the previous inequality using our definition of $\Delta_I w, \Delta_I\sigma$ listed above as follows
$$
\mu(I)\,B_Q(a)- \sum_{i=1}^M \mu(s_i(I))\,B_Q(b_i) \ge c\,c_1\,( \La w\Ra_{\mu, I} \La \sigma\Ra_{\mu, I})^{\al}\bigg(\frac{(\Delta_I w)^2}{\La w\Ra_{\mu, I}^2} +\frac{(\Delta_I\sigma)^2}{\La \sigma\Ra_{\mu, I}^2} \bigg)\mu(I)\,.
$$
Notice that $B_Q(a)=\La w\Ra_{\mu, I}\La\sigma\Ra_{\mu,I}$. Now we iterate the above inequality and get for any of Christ's dyadic $I$'s:
$$
\sum_{J\subset I\,, J\in\cD} ( \La w\Ra_{\mu, J} \La \sigma\Ra_{\mu, J})^{\al}\bigg(\frac{(\Delta_J w)^2}{\La w\Ra_{\mu, J}^2} +\frac{(\Delta_J\sigma)^2}{\La \sigma\Ra_{\mu, J}^2} \bigg)\mu(J) \le C\, Q^{\al}\mu(I)\,.
$$
This is exactly the Carleson property of the measure $\{\mu_I\}$ indicated in our Lemma \ref{uval}, with Carleson constant $C\,Q^{\al}$. The proof showed that $C$ depended only on $\al\in (0, 1/2)$ and on the doubling constant of measure $\mu$.

\end{proof}

Now, using this lemma, we start to estimate our $S_L$'s and $R_L$'s. For $S_L(\psi\sigma)$ we already had  estimate \eqref{sl}.

To estimate $R_L(\phi w)$ let us denote by $\cP_L$  maximal stopping intervals $K\in \cD, K\subset L$, where the stopping criteria are 1) either $\frac{|\Delta_K w|}{\La w\Ra_{\mu, K}} \ge \frac{1}{m+n+1}$, or  $\frac{|\Delta_K \sigma|}{\La \sigma\Ra_{\mu, K}} \ge \frac{1}{m+n+1}$, or
2) $g(K)= g(L)+m$.

\begin{lm}
\label{sbor}
If $K$ is any stopping interval then
\begin{equation}
\label{K}
\sum_{I\subset K, \ell(I)=2^{-m} \ell(L) } |\La \phi  w\Ra_{\mu, I} | \frac{|\Delta_I w|}{\La w\Ra_{\mu, I}}\frac{\mu(I)}{\sqrt{\mu(L)}} \le 2e^{\al} (m+n+1) \La |\phi | w\Ra_{\mu, K} \frac{\sqrt{\mu(K)}}{\sqrt{\mu(L)}}\sqrt{\mu_K}\La w\Ra_{\mu, L}^{-\al/2} \La \sigma\Ra_{\mu, L}^{-\al/2}\,.
\end{equation}
\end{lm}

\begin{proof}
If we stop by the first criterion, then
$$
\sum_{I\subset K, \ell(I)=2^{-m} \ell(L) } |\La \phi w\Ra_{\mu, I} | \frac{|\Delta_I w|}{\La w\Ra_{\mu, I}}\frac{\mu(I)}{\sqrt{\mu(L)}} \le 2\sum_{I\subset K, \ell(I)=2^{-m} \ell(L) } |\La \phi w\Ra_{\mu, I} | \mu(I) \frac{1}{\mu(K)}\frac{\mu(K)}{\sqrt{\mu(L)}} \le 2\,\La | \phi | w\Ra_{\mu, K}  \frac{\mu(K)}{\sqrt{\mu(L)}}
$$
$$
\le 2(m+n+1) \La |\phi | w\Ra_{\mu, K}\bigg( \frac{|\Delta_K w|}{\La w\Ra_{\mu, K}} + \frac{|\Delta_K \sigma|}{\La \sigma\Ra_{\mu, K}} \bigg)\frac{\mu(K)}{\sqrt{\mu(L)}}\le  2(m+n+1) \La |\phi | w\Ra_{\mu, K} \frac{\sqrt{\mu(K)}}{\sqrt{\mu(L)}}\sqrt{\mu_K}\La w\Ra_{\mu, K}^{-\al/2} \La \sigma\Ra_{\mu, K}^{-\al/2}\,.
$$
Now  replacing $\La w\Ra_{\mu, K}^{-\al/2} \La \sigma\Ra_{\mu, K}^{-\al/2}$ by $\La w\Ra_{\mu, L}^{-\al/2} \La \sigma\Ra_{\mu, L}^{-\al/2}$ does not grow the estimate by more than $e^{\al}$ as all pairs of son/father intervals  larger than $K$ and smaller than $L$ will have there averages compared by constant at most $1\pm \frac1{m+n+1}$. And there are at most $m$ such intervals between $K$ and $L$.

If we stop by the second criterion, then $K$ is one of $I$'s, $g(I)=g(L)+m$, and
$$
\ |\La \phi w\Ra_{\mu, I} |\frac{|\Delta_I w|}{\La w\Ra_{\mu, I}}\frac{\mu(I)}{\sqrt{\mu(L)}} \le | \La \phi w\Ra_{\mu, K} | \frac{\mu(K)}{\sqrt{\mu(L)}}\frac{|\Delta_K w|}{\La w\Ra_{\mu, K}}\le  \La |\phi | w\Ra_{\mu, K}  \frac{\sqrt{\mu(K)}}{\sqrt{\mu(L)}}\sqrt{\mu_K}\La w\Ra_{\mu, K}^{-\al/2} \La \sigma\Ra_{\mu, K}^{-\al/2}\,.
$$
Now  we replace $\La w\Ra_{\mu, K}^{-\al/2} \La \sigma\Ra_{\mu, K}^{-\al/2}$ by $\La w\Ra_{\mu, L}^{-\al/2} \La \sigma\Ra_{\mu, L}^{-\al/2}$  as before.

\end{proof}

Now
$$
R_L(\phi w) \le C(m+n+1) \La w\Ra_{\mu, L}^{-\al/2} \La \sigma\Ra_{\mu, L}^{-\al/2}\sum_{K\, \in \cP_L}  \La |\phi | w\Ra_{\mu, K} \frac{\sqrt{\mu(K)}}{\sqrt{\mu(L)}}\sqrt{\mu_K}
$$
$$
\le  C(m+n+1) \La w\Ra_{\mu, L}^{-\al/2} \La \sigma\Ra_{\mu, L}^{-\al/2} \bigg(\sum_{K\,  \in \cP_L}  \La |\phi | w\Ra_{\mu, K}^2 \frac{{\mu(K)}}{\mu(L)}\bigg)^{1/2} (\wt\mu(L))^{1/2}\,,
$$
where
$$
\wt\mu_L =\sum_{K\,  \in \cP_L} \mu_K\,.
$$
Notice that $\wt\mu_L$ form a Carleson measure with constant at most $C(m+1) Q^{\al}$.

Now we make a trick! We will estimate the right hand side  as
$$
R_L(\phi w) \le C(m+n+1)\La w\Ra_{\mu, L}^{-\al/2} \La \sigma\Ra_{\mu, L}^{-\al/2} \bigg(\sum_{K\,  \in \cP_L}  \La |\phi |w\Ra_{\mu, K}^p \frac{{\mu(K)}}{\mu(L)}\bigg)^{1/p} (\wt\mu_L)^{1/2}\,,
$$
where $p=2-\frac1{m+n+1}$. In fact,
$$
\bigg(\sum_{K\subset L, \,K\, is \,\, maximal}  \La | \phi |w\Ra_{\mu, K}^2 \frac{{\mu(K)}}{\mu(L)}\bigg)^{p/2}\le \sum_{K\,  \in \cP_L}  \La |\phi |w\Ra_{\mu, K}^p\bigg(\frac{{\mu(K)}}{\mu(L)}\bigg)^{p/2}\,.
$$
But if if $0\le j\le m$, then $(C^{-j})^{-\frac1{m+n+1}}\le C$, and therefore in the formula above $\bigg(\frac{{\mu(K)}}{\mu(L)}\bigg)^{1-\frac1{2(m+n+1)}} \le C\, \frac{{\mu(K)}}{\mu(L)}$, and $C$ depends only on the doubling constant of $\mu$. So the trick is justified.
Therefore, using Cauchy inequality, one gets
$$
R_L(\phi w) \le C(m+n+1)\La w\Ra_{\mu, L}^{-\al/2} \La \sigma\Ra_{\mu, L}^{-\al/2} \bigg(\sum_{K\,  \in \cP_L}  \La |\phi |^p w\Ra_{\mu, K}\La w\Ra_{\mu, K}^{p-1} \frac{{\mu(K)}}{\mu(L)}\bigg)^{1/p} (\wt\mu_L)^{1/2}\,.
$$
We can replace all $ \La w\Ra_{\mu, K}^{p-1}$ by $\La w\Ra_{\mu, L}^{p-1}$ paying the price by constant. This is again because all intervals larger than $K$ and smaller than $L$ will have there averages compared by constant at most $1\pm \frac1{m+n+1}$. And there are at most $m$ such intervals between $K$ and $L$. Finally,
\begin{equation}
\label{RL}
R_L(\phi w) \le C(m+n+1)\La w\Ra_{\mu, L}^{-\al/2} \La \sigma\Ra_{\mu, L}^{-\al/2} \bigg(\sum_{K\,  \in \cP_L}  \La |\phi |^p w\Ra_{\mu, K}\frac{{\mu(K)}}{\mu(L)}\bigg)^{1/p}\La w\Ra_{\mu, L}^{1-\frac1p}  (\wt\mu_L)^{1/2}
\end{equation}

We need the standard notations: if $\nu$ is an arbitrary positive measure we denote
$$
M_{\nu}f(x):=\sup_{r>0}\frac1{\nu(B(x,r))}\int_{B(x,r)} |f(x)|\,d\nu(x)\,.
$$
In particular $M_w$ will stand for this maximal function with $d\nu =w(x)\,d\mu$.

\bigskip

From \eqref{RL} we get
\begin{equation}
\label{RL1}
R_L(\phi w) \le C(m+n+1)\La w\Ra_{\mu, L}^{1-\al/2} \La \sigma\Ra_{\mu, L}^{-\al/2} \inf_L\, M_{w}(|\phi |^p )^{1/p} (\wt\mu_L)^{1/2}
\end{equation}

Now
\begin{equation}
\label{SR}
S_L(\psi\sigma)R_L(\phi w) \le C(m+n+1)\La w\Ra_{\mu, L}^{1-\al/2} \La \sigma\Ra_{\mu, L}^{1-\al/2}\frac{\inf_L \,M_{w}(|\phi |^p )^{1/p}}{\La \sigma\Ra_{\mu, L}^{1/2}} (\wt\mu_L)^{1/2}\sqrt{\sum_{J\subset L...} |(\psi \sigma, h^\sigma_J)|^2}\,,
\end{equation}
\begin{equation}
\label{RR}
R_L(\psi\sigma)R_L(\phi w) \le C(m+n+1)\La w\Ra_{\mu, L}^{1-\al} \La \sigma\Ra_{\mu, L}^{1-\al}\inf_L \,\,M_{w}(|\phi |^p )^{1/p}\inf_L\, M_{\sigma}(|\psi |^p )^{1/p} \wt\mu_L\,.
\end{equation}

Now we use the Carleson property of $\{\wt\mu_L\}_{L\in\cD}$. We need a simple folklore Lemma.

\begin{lm}
\label{carl1}
Let $\{\al_L\}_{L\in\cD}$ define Carleson measure with intensity $B$ related to dyadic lattice $\cD$ on metric space $X$.  Let $F$ be a positive function on $X$. Then
\begin{equation}
\label{1}
\sum_L (\inf_L F)\, \al_L \le 2 B\int_{X} F\,d\mu\,.
\end{equation}
\begin{equation}
\label{2}
\sum_L\frac{ \inf_L F}{\La\sigma\Ra_{\mu, L}} \al_L \le C\,B\int_{X}\frac{F}{\sigma} d\mu\,.
\end{equation}
\end{lm}

Now use \eqref{SR}. Then  the estimate of $III\le \sum_L S_L(\psi\sigma) R_L(\phi w)$ will be reduced to estimating
$$
(m+n+1)Q^{1-\al/2}\bigg(\sum_L \frac{\inf_L M_{w}(|\phi |^p )^{2/p}}{\La \sigma\Ra_{\mu, L}} \wt\mu_L\bigg)^{1/2}\le (m+n+1)^2\,Q\,\bigg(\int_{\R} (M_{w}(|\phi |^p ))^{2/p} wd\mu\bigg)^{1/2}
$$
$$
\le  (\frac1{2-p})^{1/p}(m+n+1)^2\,Q\,\bigg(\int_{\R} \phi^2 \, wd\mu\bigg)^{1/2}\le (m+n+1)^3\,Q\,\bigg(\int_{\R} \phi^2 \, wd\mu\bigg)^{1/2} \,.
$$
Here we used \eqref{2} and the usual estimates of maximal function $M_{\mu}$ in $L^q(\mu)$ when $q\approx 1$. Of course for $II$ we use the symmetric reasoning. 

\bigskip

Now $IV$: we use \eqref{RR} first.
$$
\sum_L S_L(\psi\sigma) R_L(\phi w) \le (m+n+1)Q^{1-\al}\sum_L\inf_L\, M_{w}(|\phi |^p )^{1/p}\inf_L\, M_{\sigma}(|\psi |^p )^{1/p}\wt\mu_L
$$
$$
\le C (m+n+1)^2Q\int_{\R}(M_{w}(|\phi |^p ))^{1/p}\,(M_{\sigma}(|\psi |^p ))^{1/p}w^{1/2}\sigma^{1/2}d\mu
$$
$$
\le C (m+n+1)^2Q\,\bigg(\int_{\R}(M_{w}(|\phi |^p ))^{2/p}\, wd\mu\bigg)^{1/2}\,\bigg(\int_{\R}(M_{\sigma}(|\psi |^p ))^{2/p}\, \sigma d\mu\bigg)^{1/2}
$$
$$
\le C (m+n+1)^4\,Q\,\bigg(\int_{\R}\phi^2\, wd\mu\bigg)^{1/2}\bigg(\int_{\R}\psi^2\, \sigma d\mu\bigg)^{1/2}\,.
$$
Here we used \eqref{1} and the usual estimates of maximal function $M_{\mu}$ in $L^{2/p}(\mu)$ when $p\approx 2,\, p<2$.


\begin{thebibliography}{XXX}
\label{rf}


\bibitem{Buck1} {\sc S. M. Buckley}, {\em Estimates for operator norms on weighted spaces
and reverse Jensen inequalities},  Trans. Amer. Math. Soc., {\bf 340} (1993), no. 1, p53--272.

\bibitem{BCR} {\sc G. Beylkin, R. Coifman, V. Rokhlin}, {\em Fast wavelet transforms and numerical algorithms, I},  Comm. Pure and Appl. Math., {\bf 44} (1991), No. 2, 141--183.

\bibitem{Chr} {\sc M. Christ},  {\em A $T(b)$ theorem with remarks on analytic capacity and the Cauchy integral}, Colloq. Math. {\bf 60/61} (1990), no. 2, 601-628.

\bibitem{H} {\sc T.~Hyt\"{o}nen}, {\em The sharp weighted bound for general Calderon-Zygmund operators},  arXiv:1007.4330.

\bibitem{H1} {\sc T.~Hyt\"{o}nen}, {\em Nonhomogeneous vector $Tb$ theorem},  arXiv:0809.3097.

\bibitem{HLRSVUT} {\sc T.~Hyt\"{o}nen, M. Lacey, M. C. Reguera, E. Sawyer, A. Vagharshakyan, I. Uriarte-Tuero}, {em Weak and Strong type $ A_p$ Estimates for Calderãn-Zygmund Operators},
arXiv:1006.2530.


\bibitem{HMW} {\sc R. Hunt, B. Muckenhoupt, R. Wheeden}, {\em Weighted norm inequalities  for the conjugate function and the
Hilbert transform}, Trans. Amer. Math. Soc., {\bf 176} (1973), pp. 227-251.



\bibitem{HPTV} {\sc T. Hyt\"onen, C. P\'erez, S. Treil, A. Volberg}, {A sharp estimated of weighted dyadic shifts that gives the proof of $A_2$ conjecture}, Preprint  Sept. 2010, Preprint arXiv 1010.0755.


\bibitem{LPR}{\sc M. Lacey, S. Petermichl, M. Riguera}, {\em Sharp ${A}_2$ inequality for {H}aar shift operators}  arXiv:0906.1941.

\bibitem{KV} {\sc S. Konyagin, A. Volberg}, {\em On measures with the doubling condition}. (Russian) Izv. Akad. Nauk SSSR Ser. Mat. {\bf 51} (1987), no. 3, 666--675;
 translation in Math. USSR-Izv. {\bf 30} (1988), no. 3, 629--638.

\bibitem{Ler1} {\sc A. Lerner} {\em A pointwise estimate for local sharp maximal
                  function with applications to singular integrals},  preprint, 2009.
                  
                  
\bibitem{NV}   {\sc F. Nazarov, A. Volberg}, Bellman function,  polynomial estimates  of weighted dyadic shifts, and  $A_2$ conjecture. Preprint, Nov. 2010, pp. 1--6.              

\bibitem{NTV5} {\sc F.~Nazarov, S.~Treil, and A.~Volberg,} {\em $Tb$ theorems on nonhomogeneous spaces},  Acta Math., {\bf 190}  (2003), 151--239.

\bibitem{NTV6} {\sc F.~Nazarov, S.~Treil, and A.~Volberg,} {\em Two weight inequalities for individual Haar multipliers and other well localized operators}, Preprint 2004, 1--14. Appeared in
 Math. Res. Lett. 15 (2008), no. 3, 583--597.

\bibitem{NTV7} {\sc F.~Nazarov, S.~Treil, and A.~Volberg,}
{\em Two weight $T1$ theorem for the Hilbert transform: the case of doubling measures}, Preprint 2004, 1--40.

\bibitem{NTVlost} {\sc F.~Nazarov, S.~Treil, and A.~Volberg,}  {\em Two weight estimate for the Hilbert transform and corona decomposition for non-doubling measures}, Preprint 2005, 1-33. Put into arXive in 2010.

\bibitem{NTV-name} F.~Nazarov, S.~Treil and  A.~Volberg,
\emph{Bellman function in stochastic control and harmonic analysis.}
Systems, approximation, singular integral operators, and related topics
(Bordeaux, 2000),  393--423, Oper. Theory Adv. Appl., \textbf{129},
Birkhauser, Basel, 2001.

\bibitem{NTV-2w} {\sc F.~Nazarov, S.~Treil, and A.~Volberg,} {\em The Bellman functions and two-weight inequalities for
Haar multipliers}, J. of Amer. Math. Soc., {\bf 12}, (1999), no. 4, 909-928.

\bibitem{Petm} {\sc S. Petermichl}, {\em Dyadic shifts and a logarithmic estimate for Hankel operators with matrix symbol},
C. R. Acad. Sci. Paris, S\'er. I Math., {\bf 330}, (2000), no. 6, pp. 455-460.

\bibitem{PetmV} {\sc S. Petermichl, A. Volberg}, {\em Heating of the Ahlfors-Beurling operator: weakly quasiregular maps on the
plane are quasiregular}, Duke Math. J., {\bf 112} (2002), no. 2, pp. 281-305.

\bibitem{Petm1} {\sc S. Petermichl}, {\em The sharp bound for the Hilbert transform on weighted Lebesgue spaces in terms of the classical $A_p$ characteristic. }Amer. J. Math. 129 (2007), no. 5, 1355--1375.

\bibitem{Petm2} {\sc S. Petermichl}, {\em The sharp weighted bound for the Riesz transforms.} Proc. Amer. Math. Soc. 136 (2008), no. 4, 1237--1249.

\bibitem{Petm3} {\sc S. Petermichl}, {\em Dyadic shifts and a logarithmic estimate for Hankel operators with matrix symbol.} C. R. Acad. Sci. Paris S?r. I Math. 330 (2000), no. 6, 455--460.

\bibitem{PTV1}  {\sc C. P\'erez, S. Treil, A. Volberg}, {\em On $A_2$ conjecture and corona decomposition of weights}, arxiv1005.2630.

\bibitem{PTV2} {\sc C. P\'erez, S. Treil, A. Volberg}, {\em $A_2$ conjecture: reduction to a bilinear embedding, splines and less smooth \cz kernels}, Preprint,  pp. 1--17, July 3, 2010.

\bibitem{RTV} {\sc A. Reznikov, S. Treil, A. Volberg}, {\em A sharp estimate of weighted dyadic shifts of complexity $0$ and $1$}, Preprint, September 2010, pp. 1--9.

\bibitem{RV} {\sc A. Reznikov, A. Volberg}, {\em Random ``dyadic'' lattice in geometrically doubling metric space and $A_2$ conjecture}, Preprint arXiv:1103.5246, pp. 1--10.

\bibitem{VaVo1} {\sc V. Vasyunin, A. Volberg}, {\em Two weight inequality: the case study}, St. Petersburg Math. J., 2005?

\bibitem{Vo1} {\sc A. Volberg}, {\em \cz capacities and operators on non-homogeneous spaces}, CBMS conferences, AMS, vol. 100, 2003, 165pp.


\bibitem{Wit}{\sc J. Wittwer}, {\em A sharp estimate on the norm of the martingale transform.} Math. Res. Lett. 7 (2000), no. 1, 1--12.



\end{thebibliography}
\end{document}